\documentclass{article}

\usepackage{tikz}
\usepackage[utf8]{inputenc} % allow utf-8 input
\usepackage[T1]{fontenc}    % use 8-bit T1 fonts
\usepackage{url}            % simple URL typesetting
\usepackage{booktabs}       % professional-quality tables
\usepackage{amsfonts}       % blackboard math symbols
\usepackage{nicefrac}       % compact symbols for 1/2, etc.
\usepackage{microtype}      % microtypography
\usepackage{lipsum}		% Can be removed after putting your text content
\usepackage{enumitem}
\usepackage{graphicx}
\usepackage[square,numbers]{natbib}
\usepackage{doi}
\usepackage{hyperref}
\usepackage{amsmath,amssymb,amsthm,eucal, graphicx, units}
\usepackage{float}
\newtheorem{theorem}{Theorem}[section]
\newtheorem{corollary}{Corollary}[theorem]
\newtheorem{lemma}[theorem]{Lemma}
\newtheorem{proposition}[theorem]{Proposition}

\title{The Restriction of Efficient Geodesics to the Non-separating Complex of Curves}
\date{}
\author{ Seth  Hovland and Greg Vinal}

\hypersetup{
pdftitle={Restriction of Efficient Geodesics to $N(C)$},
}

\begin{document}
\maketitle

\begin{abstract}
 In the complex of curves of a closed orientable surface of genus $g,$ $\mathcal{C}(S_g),$ a preferred finite set of geodesics between any two vertices, called \emph{efficient geodesics} introduced by Birman, Margalit, and Menasco in \cite{birman_margalit_menasco_2016}. The main tool used to establish the existence of efficient geodesics was a \emph{dot graph}, which recorded the intersection pattern of a reference arc with the simple closed curves associated with a geodesic path.  The idea behind the construction was that a geodesic that is not initially efficient contains shapes in its corresponding dot graph. 
 These shapes then correspond to surgeries that reduce the intersection with the reference arc.  In this paper, we show that the efficient geodesic algorithm is able to be restricted to the non-separating curve complex; the proof of this will involve analysis of the dot graph and its corresponding surgeries.  Moreover, we demonstrate that given any geodesic in the complex of curves we may obtain an efficient geodesic whose vertices, with the possible exception of the endpoints, are all nonseparating curves.
\end{abstract}

% keywords can be removed

\section{Introduction}
\subsection{The Complex of Curves and Geodesics }
The complex of curves $\mathcal{C}(S)$ for a compact surface $S$ is a simplicial complex whose vertices correspond to isotopy classes of essential simple closed curves in $S$ and whose edges connect vertices with disjoint representatives.  We can endow the 0-skeleton of $\mathcal{C}(S)$ with metric by defining the distance between two vertices, $u,v$ to be the minimal number of edges in any path between them.  In this paper, as in \cite{birman_margalit_menasco_2016}, we assume that the surfaces we are considering are closed and have genus at least two.  It is a fundamental result that in this case, $\mathcal{C}(S)$ is connected.  Thus, the distance is defined for all pairs of vertices in $\mathcal{C}(S).$  The trouble is that the complex of curves is, in fact, too connected.  It turns out that $\mathcal{C}(S)$ is locally infinite (for any vertex $v$ there are infinitely many adjacent vertices $w$) and there are infinitely many geodesics between most pairs of vertices.  Thus, it is useful to have a preferred finite subset of geodesics to choose from.  This is the idea behind the introduction of \emph{tight geodesics} in \cite{masur_minsky_2000}. In \cite{birman_margalit_menasco_2016}, Birman, Margalit, and Menasco introduced an alternate preferred finite set of of geodesics, called \emph{efficient geodesics}.  The novel feature of this particular set of geodesics is the algorithm used to generate them.  The algorithm can (and has been \cite{GLENN2017115}) implemented on a computer to find efficient geodesics for small distances. \\

We let $\mathcal{N}(S)$ denote the subcomplex of $\mathcal{C}(S)$ spanned by vertices corresponding to non-separating simple closed curves.  This complex is called the \emph{complex of nonseparating curves.}  Again, it is a standard result that if the genus of $S$ is at least two, $\mathcal{N}(S)$ is connected.  We show below that given a geodesic in the complex of curves we can always find a geodesic that is restricted to the subcomplex of nonseparating curves. It is then natural ask whether efficient geodesics may also be restricted to this subcomplex.

\subsection{Efficient Geodesics}
The idea behind obtaining an efficient geodesic $v_0,\dots,v_n$ in $\mathcal{C}(S)$ is to iteratively decrease intersections with an arc as we move along the path. We explain further below.  The following construction and results were first introduced in \cite{birman_margalit_menasco_2016}, the reader familiar with these results is invited to see section 2. \\

Suppose that $\gamma$ is an arc in $S$ and $\alpha$ is a simple closed curve in $S.$  Then we say that $\gamma$ and $\alpha$ are in \emph{minimal position} if $\alpha$ is disjoint from the endpoints of $\gamma$ and the number of points of intersection of $\alpha$ and $\gamma$ is smallest over all simple closed curves homotopic to $\alpha$ through homotopies that do not pass through the endpoints of $\gamma.$ \\

    Let $v_0,\dots,v_n$ be a geodesic of length at least three in the complex of curves, and let $\alpha_0,\alpha_1,$ and $\alpha_n$ be representatives of $v_0,v_1,$ and $v_n$ that are pairwise in minimal position. A \emph{reference arc} for the triple $\alpha_0,\alpha_1,\alpha_n$ is an arc $\gamma$ that is in minimal position with $\alpha_1$ and whose interior is disjoint from $\alpha_0\cup\alpha_n.$\\
        We say that the oriented geodesic $v_0,\dots,v_n$ is \emph{initially efficient} if 
        $$
        \vert\alpha_1\cap\gamma\vert\leq n-1
        $$
for all choices of reference arcs $\gamma.$  Finally, we say that $v=v_0,\dots,v_n=w$ is \emph{efficient} if the oriented geodesic $v_k,\dots,v_n$ is initially efficient for each $0\leq k\leq n-3$ and the oriented geodesic $v_n,v_{n-1},v_{n-2},v_{n-3}$ is also initially efficient. Thus, to test the efficiency of a geodesic we look at all the triples $v_k,v_{k+1},v_n$ and count the intersection of $v_{k+1}$ with any reference arc. While it may seem impossible to check intersections with \emph{all} reference arcs, it turns out that there are finitely many of them.  Moreover, in special cases it is sufficient to check the intersections of $\alpha_i\cap\alpha_n$ for $1\leq i\leq n-1$ \cite{birman_margalit_menasco_2016}.\\

Given a vertex path $v_0,\dots, v_n$ in $\mathcal{C}(S)$ with representative curves $\alpha_0,\dots,\alpha_n$ and an oriented reference arc $\gamma$, we may traverse $\gamma$ in the direction of its orientation and record the order in which the curves $\alpha_0,\dots,\alpha_n$ intersect $\gamma.$  The result is a sequence of natural numbers $\sigma\in \{1,\dots,n-1\}^N.$ Where $N$ is the minimal cardinality of $\gamma\cap(\alpha_1\cup\cdots\cup\alpha_{n-1}).$ The sequence $\sigma$ is called the \emph{intersection sequence} of the $\alpha_i$ along $\gamma.$ \\

The \emph{complexity} of an oriented path $v_0,\dots,v_n\in \mathcal{C}(S)$ is defined to be 
$$
\sum_{k=1}^{n-1}(i(v_0,v_k)+i(v_k,v_n)).
$$
We say that a sequence $\sigma$ of natural numbers is \emph{reducible} under the following circumstances: whenever $\sigma$ arises as an intersection sequence for a path $v_0,\dots,v_n$ in $\mathcal{C}(S)$ there is another path $v_0',\dots,v_n'$ with smaller complexity.  This allows us restate the existence of initally efficient paths in terms of this complexity measurement. 
\begin{proposition}
    Suppose $\sigma$ is a sequence of elements of $\{1,\dots,n-1\}.$ If $\sigma$ has more than $n-1$ entries equal to 1, then $\sigma$ is reducible.
\end{proposition}
From the above proposition we can deduce the existence of initially efficient geodesics. 
\begin{proposition}
    Let $g\geq 2.$ If $v$ and $w$ are vertices of $\mathcal{C}(S),$ with $d(v,w)\geq 3,$ then there exists an initially efficient geodesic from $v$ to $w.$
\end{proposition}

We note that the above definitions and propositions all follow when restricted to the nonseparating curve complex.

\subsection{Sawtooth form and the dot graph}
The proof of Proposition 1.1 was carried out in three stages. First, the intersection sequence was put into a normal form.  This is called \emph{sawtooth form}.  Then, associated to the sawtooth form for the sequence is a diagram called the \emph{dot graph}.  The reduciblility of an intersection sequence then corresponds to certain geometric features in the dot graph. We review these now. 

We may exchange the order of intersection of two curves that are adjacent in the intersection sequence by performing a \emph{commutation} as described in \cite{birman_margalit_menasco_2016}.  The result is a sequence that is in \emph{sawtooth form}.  That is, we say a sequence $(j_1,j_2,\dots, j_k)$ of natural numbers is in \emph{sawtooth form} if $$
j_i<j_{i+1}\Rightarrow j_{i+1}=j_i+1
$$   
An example of a sequence in sawtooth form is $(1,2,2,3,4,3,4,2,3,4,5).$  Given a sequence of natural numbers in sawtooth form, we also consider its \emph{ascending sequences}, these are the maximal subsequences of the form $k,k+1,\dots,k+m.$  In the above example the ascending sequences are $(1,2),(2,3,4),(3,4),(3,4), (2,3,4,5).$  It is clear that if we have an intersection sequence and we perform a finite number of commutations we may get the intersection sequence into sawtooth form while keeping the number of intersections of the $\alpha_i$'s with $\gamma$ constant.\\

Next, given an intersection sequence $\sigma$ in sawtooth form, we may regard it as a function ${1,\dots,N}\to\mathbb{N}$ and plot it in $\mathbb{R}^2_{\geq 0}.$  The points of the graph of a sequence will be called \emph{dots.}  We decorate the graph by connecting the dots that lie on a a given line of slope 1; these line segments are called \emph{ascending segments.}  The resulting decorated graph is called the \emph{dot graph} of $\sigma$ and is denoted $G(\sigma).$  See figure \ref{fig:Dotgraphexample}. Again, the idea behind this construction is that given a geodesic that is not efficient we can see shapes in its corresponding dot graph that correspond to surgeries that reduce the intersection with the reference arc.\\
  \begin{figure}[h!]
        \centering
        \includegraphics[scale=0.8]{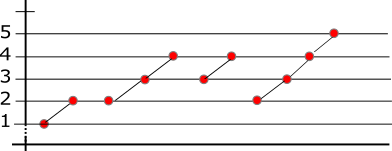}
        \caption{The Dot Graph of $(1,2,2,3,4,3,4,2,3,4,5)$}
        \label{fig:Dotgraphexample}
    \end{figure}

\subsection{Dot graph polygons and surgery}
This section is a summary of section 3.3 in \cite{birman_margalit_menasco_2016}.  We first review the surgeries described, then in the next section we discuss the results of these surgeries when restricted to $\mathcal{N}(S).$ \\

After putting an intersection sequence into sawtooth form and constructing its dot graph, it was shown that if the dot graph contained certain geometric shapes, it corresponded to a sequence that was reducible.  These shapes were called \emph{dot graph polygons.}  In particular, the existence of a \emph{box}, \emph{hexagon of type I} or  \emph{hexagon of type II} in the dot graph (shown in Figure \ref{fig:dotgraphpolys}, respectively), implied that the sequence $\sigma$ was reducible. To remove these shapes from the dot graph, \emph{surgeries} on the curves in the intersection pattern corresponding to these shapes were introduced. \\
  \begin{figure}[h!]
        \centering
        \includegraphics[width=\linewidth,scale=0.8]{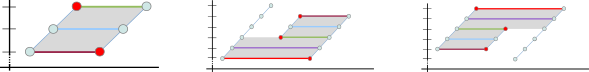}
        \caption{Dot Graph Polygons (Box, Hexagon Type 1, and Hexagon Type 2)}
        \label{fig:dotgraphpolys}
    \end{figure}

We do surgery on a curve $\alpha$ that intersects our intersection arc $\gamma$ in at least two points.  We draw a neighborhood of $\gamma$ so that it is horizontal and oriented to the right. We then remove from $\alpha$ small neighborhoods of its points of intersection with $\gamma,$ this results in a pair of curves.  We will then join two of the endpoints back together forming a new simple closed curve, and discard the other curve. Depending on how we join pairs of endpoints,  we say that $\alpha'$ is obtained from $\alpha$ by $++,$ $+-,$ $-+,$ or $--$ surgery along $\gamma.$ The first symbol is $+$ or $-$ depending on whether the first endpoint of $\alpha$ lies to the left or right of $\gamma,$ respectively. Similarly for the second symbol. When we are considering an arbitrary simple closed curve, exactly two of the four possible surgeries result in a simple closed curve. If we give $\alpha$ an orientation then two intersection points of $\alpha$ and $\gamma$ can either agree or disagree in orientation.  If the agree, then the $+-$ and $-+$ surgeries, which are called \emph{odd},
 result in a simple closed curves.  Otherwise, the $++$ and $--$ surgeries, called \emph{even} surgeries result in a simple closed curve.\\

  \begin{figure}[h!]
        \centering
        \includegraphics[width=\linewidth,scale=0.8]{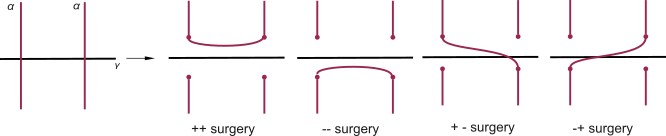}
        \caption{Surgery on a curve}
        \label{fig:surgtypes}
    \end{figure}

Suppose that we have a geodesic in $\mathcal{C}(S)$ containing $\alpha$ as a representative for some vertex $v_i.$ If we are to perform surgery on $\alpha$ then it must intersect our  intersection arc $\gamma$ at least twice (otherwise it stays in the geodesic and is not replaced).   Thus, $\alpha$ either intersects $\gamma$ consecutively, or between these intersections $\gamma$ intersects with at least one other curve $\beta.$ In the complex of curves, we can immediately get rid of the first case, where $\alpha$ intersects $\gamma$ consecutively by performing the surgeries described above. This does not follow so easily in the subcomplex of nonseparating curves (See Proposition \ref{Prop2.6}).  However, if $\alpha$ does not intersect $\gamma$ consecutively, as is the case when we see boxes, and hexagons in the dot graph, we will see that we have the same choice of surgeries as before. \\

Using the above surgeries, it was shown in \cite{birman_margalit_menasco_2016} that a dot graph with an empty, unpierced box or an empty, unpierced hexagon of type 1 or 2 corresponds to a sequence that is reducible. This was done by prescribing a sequence of surgeries that replaced the $\alpha_i$ curves with new $\alpha_i'$ that resulted in a path with smaller complexity. Below we state the required sequence of surgeries corresponding to the type of dot graph polygon.\\

Suppose the dot graph $G(\sigma)$ has an empty, unpierced box $P.$ Then the corresponding sequence of intersections along $\gamma$ has the form: 
 $$\alpha_k,\dots,\alpha_{k+m},\alpha_k,\dots,\alpha_{k+m}$$
where $1\leq k\leq k+m\leq n-1.$ For the vertices not in $\{ k,\dots, k+m\}$, they remain unchanged.  We define $\alpha_k', \dots,\alpha_{k+m}'$ inductively: for $i=k,\dots,k+m$ the curve $\alpha_i'$ is obtained by performing surgery along $\gamma$ between the two points of $\alpha_i\cap \gamma$ corresponding to dots of $P$ the surgeries are chosen so that they form a path in the directed graph below.  
 \begin{figure}[h!]
        \centering
        \includegraphics{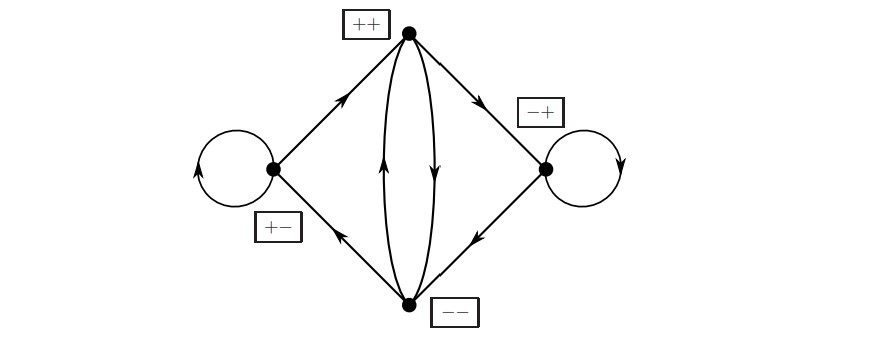}
        \label{fig:directedgraph}
    \end{figure}
The vertices of the graph correspond to the four types of surgeries described above, the rule is that the second sign of the origin of a directed edge is opposite of the first sign of the terminus. It is clear from the graph that the desired sequence of surgeries exists.  We demonstrate this procedure below, where we perform $-+$ surgery on $\alpha_3,$ then $--$ surgery on $\alpha_4$ and finally $+-$ surgery on $\alpha_5.$  It is an easy check that replacing the curves $\alpha_i$ with these new ones results in a reduced intersection sequence $\sigma.$ 

\begin{figure}[h!]
        \centering
        \includegraphics[width=\linewidth, height=6.5cm]{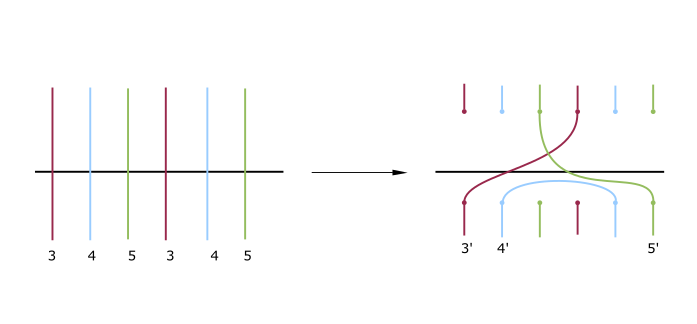}
        \label{fig:surgerybox}
    \end{figure}

 Suppose the dot graph $G(\sigma)$ has an empty, unpierced hexagon $P$ of type 1. The case of a type 2 hexagon is nearly identical so it will be omitted.  By definition of sawtooth form and of a type 1 hexagon, there are no ascending segments of $G(\sigma)$ in the vertical strip between the leftmost and middle ascending edges of $P$ and any ascending segments of $G(\sigma)$ that lie in the vertical strip between the middle and rightmost ascending segments have their highest point strictly below the lower-right horizontal edge of $P.$ See Figure \ref{fig:hexpoly} below: 
  \begin{figure}[h!]
        \centering
        \includegraphics[scale=0.8]{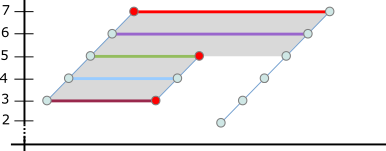}
        \caption{Type 1 Hexagon}
        \label{fig:hexpoly}
    \end{figure}
    
 It follows that the dots of $P$ correspond to a sequence of intersections along $\gamma$ of the following form 

$$\alpha_k,\dots,\alpha_{k+m},\alpha_k,\dots,\alpha_{k+l},\alpha_{j_1},\dots,\alpha_{j_p},\alpha_{k+l},\dots,\alpha_{k+m}$$
where $1\leq k\leq k+l\leq k+m\leq n-1,p\geq 0,$ and each $j_i<\alpha_{k+l}.$  See the above figure where $k=3,l=2,m=4,$ and $p=0.$
  
As in the case with the box, whenever we have $\alpha_i$ with $i\not\in\{k,\dots,k+m\}$ we set $\alpha_i'=\alpha_i.$  Each of the remaining $\alpha_i$ correspond to two dots in $P$ except for $\alpha_{k+l}$ which corresponds to three. Let $\alpha_{k+l}'$ be the curve obtained from $\alpha_{k+l}$ via surgery along $\gamma$ between the first two (leftmost) points of $\alpha_{k+l}\cap \gamma$ corresponding to dots of $P$ and satisfying the following property: $\alpha_{k+l}'$ does not contain the arc of $\alpha_{k+l}$ containing the third (rightmost) point of $\alpha_{k+l}\cap \gamma$ corresponding to a dot on $P.$  We then define $\alpha_{k+l-1}',\dots,\alpha_k'$ inductively as before using the directed graph above, finally, we define $\alpha_{k+l+1}',\dots,\alpha_{k+m}'$ inductively as before. It is readily verified that this procedure reduces $\sigma.$
 \begin{figure}[h!]
        \centering
        \includegraphics[width=\linewidth,scale=0.8]{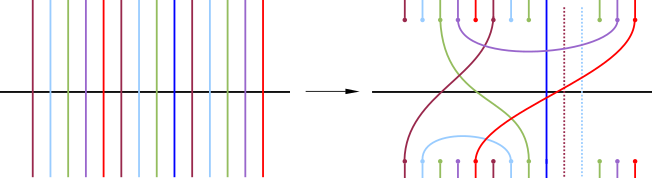}
        \caption{The surgery described above on a type 1 hexagon}
        \label{fig:hexsurg}
    \end{figure}

 Using these surgeries to remove from the dot graph the above polygons results in an initially efficient geodesic.  The last step is to inductively produce initially efficient geodesics for the triples $v_k,\dots,v_n$  for each $0\leq k\leq n-3$ and the oriented geodesic $v_n,v_{n-1},v_{n-2},v_{n-3}.$ This was done in Section 3.5 of \cite{birman_margalit_menasco_2016}. The exact inductive argument works when restricting to $\mathcal{N}(S),$ so our work lies solely in showing the construction of initially efficient geodesics restricts to $\mathcal{N}(S).$

\section{Existence of Efficient Geodesics in Complex of Non-Separating Curves}

Our main result is that efficient geodesics exist in the complex of nonseparating curves.  Since this is a subcomplex of the complex of curves, the fact that there are finitely many of them follows from Theorem 1.1 in \cite{birman_margalit_menasco_2016}.

\begin{theorem}
Let $g\geq 2.$  If $v$ and $w$ are vertices of $\mathcal{N}(S_g)$ with $d(v,w)\geq 3,$ then there exists an efficient geodesic from $v$ to $w$ in $\mathcal{N}(S_g).$   Additionally, there are finitely many efficient geodesics from $v$ to $w.$
\end{theorem}

\subsection{Proof of Theorem 2.1}
The result lies in following exactly the proof setup for the complex of curves in \cite{birman_margalit_menasco_2016}.  We prove the existence of initially efficient geodesics in $\mathcal{N}(S)$ (see proposition below).  Then the additional inductive step will follow exactly as outlined in Section 3.5 of \cite{birman_margalit_menasco_2016}.  The key observation of this paper is that Lemma \ref{Lemma 2.2} below holds when restricting to $\mathcal{N}(S)$:

\begin{lemma}\label{Lemma 2.2}
    Suppose that $\sigma$ is a sequence of natural numbers in sawtooth form and that $G(\sigma)$ has an empty, unpierced box or an empty, unpierced
    hexagon of type 1 or 2. Then $\sigma$ is reducible (in $\mathcal{N}(S))$. 
\end{lemma}
This will allow us to prove:
\begin{proposition}
    Let $g\geq 2.$  If $v$ and $w$ are vertices of $\mathcal{N}(S)$ with $d(v,w)\geq 3,$ then there exists an initially efficient geodesic from $v$ to $w$ in $\mathcal{N}(S).$
\end{proposition}
In \cite{birman_margalit_menasco_2016} the existence of the above polygons in the dot graph came with surgeries on the curves representing vertices in the given geodesic that removed these shapes in the dot graph.  It is possible to replace each new curve with with a nonseparating curve, most of the time the argument follows exactly the same.  We need only to show that this is the case, and to treat the few outlying examples. 

We begin with the well-known fact that a separating curve must intersect any other curve an even number of times. 
\begin{proposition}\label{thebigprop}
     Let $\alpha$ be a simple closed curve in $S.$  If $\alpha$ separates $S$ into two components, then for any simple closed curve $\beta$ we have the geometric intersection number, $i(\alpha,\beta)$, is even.
 \end{proposition}
\begin{proof}
    Let $S_\alpha$ be the surface that results from splitting $S$ along $\alpha.$   In $S_\alpha,$ the curve $\beta$ is either unchanged (all of $\beta$ is in one of the connected components of $S_\alpha$) or  $\beta$ is a collection of arcs with endpoints along $\alpha.$ Each of these arcs has two endpoints.  Thus intersections, should they exist, come in pairs. In either case this gives an even number of intersections with $\alpha.$  
\end{proof}
The quick test we will use when deciding if a curve made via surgery is nonseparating is to see if $i(\alpha',\beta)=1$ for some curve $\beta$. Next, we must show that given a geodesic in the complex of curves, we may take all the vertices in the path to be nonseparating curves.  This allows us to start with geodesic in $\mathcal{N}(S)$ and attempt to make it efficient. \\

\begin{proposition}\label{geoinN(S)}
    Given a geodesic in $\mathcal{C}(S)$ with endpoints $v,w\in\mathcal{N}(S)$ there exists a geodesic from $v$ to $w$ with each vertex a nonsepartating curve.   
\end{proposition}
\begin{proof}
Let $v=v_0,\dots,v_n=w$ be a geodesic in $\mathcal{C}(S).$  If all the vertices in this geodesic are nonseparating curves then we are done.  Assume that $v_i$ is a separating curve in the above geodesic with lowest index $i$.  Then consider the subpath $v_{i-1},v_i,v_{i+1}$ because $v_i$ is separating it divides the surface $S$ into two components.  Both $v_{i-1}$ and $v_{i+1}$ are disjoint from $v_i$ however since there is not an edge between them they intersect each other.  Therefore, they are both in one of the connected components of $S_v.$  In the other connected component choose a nonseparating simple closed curve $v_i'.$ Such a curve exists, as otherwise $v_i$ would be inessential.  This curve is disjoint from both $v_{i-1}$ and $v_{i+1}$ and may replace $v_i$ in the geodesic. Continuing in this way gives a geodesic in the subcomplex of nonseparating curves.    
\end{proof}

\begin{corollary}
    $\mathcal{N}(S)$ is connected.
\end{corollary}
\begin{proof}
   Since $\mathcal{C}(S)$ is connected. Replacing each separating curve with a nonseparating one gives a path in $\mathcal{N}(S).$ 
\end{proof}

\subsection{The Trivial Surgeries}
The above proposition now allows us to start with a geodesic in $\mathcal{N}(S),$ we wish to now do simplifying surgeries on it.  We begin with the case where our reference arc $\gamma$ sees a curve $\alpha$ consecutively. In $\mathcal{C}(S),$ we performed an even or odd surgery depending on the orientation of $\alpha$ and were guaranteed that each resulted in a simple closed curve.  However, when we restrict to the subcomplex of nonseparating curves we are no longer guaranteed that both curves are nonseparating.  For instance, take a separating curve and connect sum it with a nonseparating curve.  Then performing surgery along an arc separates the curve into a separating curve and an nonseparating one. The proposition below however shows that performing surgery on a nonseparating curve that meets a reference arc $\gamma$ consecutively, will always yield at least one nonseparating curve.  

\begin{proposition}\label{Prop2.6}
   Let $\gamma$ be a reference arc, and $\alpha$ be a nonsepartaing simple closed curve. Suppose that $i(\alpha,\gamma)\geq 2.$  Then there exists a simple closed curve  $\alpha'$ obtained from the nonseparating simple closed curve $\alpha$ via some surgery along $\gamma$ that is still nonseparating. 
 \end{proposition}
\begin{proof}
    Orient $\gamma$ and $\alpha.$  Consider two points of intersection that are consecutive along $\gamma.$  The orientation of $\gamma$ and $\alpha$ allow us to assign an index to each intersection either $+1$ or $-1.$
    If two points of intersection have the same index, we preform an odd surgery.  The resulting curve $\alpha'$ crosses $\gamma$ one time and intersects $\alpha$ exactly once. (See Figure \ref{fig:oddsurgrep} below for +- surgery). We emphasize that we construct $\alpha'$ so that outside the local picture below, $\alpha'$ lies just to the right of $\alpha.$  Since $i(\alpha,\alpha')=1,$ $\alpha'$ is nonseparating. 

\begin{figure}[h!]
        \centering
        \includegraphics[width=3.5in]{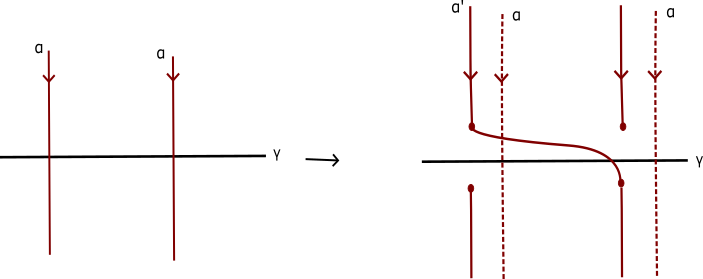}
        \caption{Odd surgery always results in a nonseparating curve}
        \label{fig:oddsurgrep}
    \end{figure}

    If the two intersection points have opposite indices we perform an even surgery.  In this case, we may need to make a choice of curve to replace $\alpha$ with, since one of the surgeries may give a nonsepartaing curve.  One of the curves remains above $\gamma$ and the other remains below $\gamma$.  We argue that at least one of these curves is nonseparating. 

    \begin{figure}[h!]
        \centering
        \includegraphics[scale=0.8]{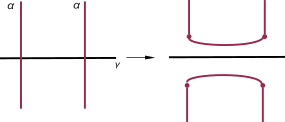}
        \caption{++ and $- -$ surgery on $\alpha$}
        \label{fig:evensurgoptions}
    \end{figure}

    Assume for contradiction that both curves were separating this would divide $S$ into three connected components. Then joining the curves back together would give back our original curve $\alpha$ however $\alpha$ would still separate our surface.  
    
    \begin{figure}[h!]
        \centering
        \includegraphics[scale=0.8]{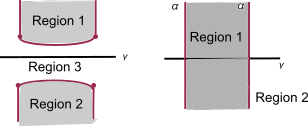}
         \caption{Figure for Proof of Propostion \ref{Prop2.6}}
        \label{fig:evensurgoptions}
    \end{figure}
    
     Thus, performing an even surgery always results in at least one nonseparating curve.
\end{proof}

\subsection{The Non-Trivial Surgeries}
We use the dot graph exactly as in \cite{birman_margalit_menasco_2016} to determine how to reduce our intersections sequence.  When restricting to the nonseparating curve complex we need to show that each dot graph polygon has a surgery that results in the removal of the polygon and whose new curves are all non-separating.  This will prove Lemma \ref{Lemma 2.2}.   

Throughout we assume that $\sigma$ is an intersection sequence of nonseparating curves in sawtooth form.\\
  \begin{figure}[h!]
        \centering
        \includegraphics[scale=0.8]{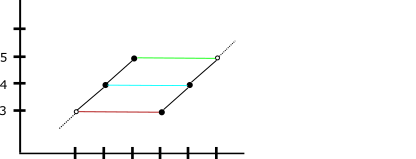}
        \caption{An empty, unpierced box}
        \label{fig:boxsurgeNC}
    \end{figure}

{\bf Case 1:} Suppose that $\sigma$ is a sequence of natural numbers in sawtooth form and that $G(\sigma)$ has an empty, unpierced box.  Then $\sigma$ is reducible. \\

This is the easy case.  Carry out the surgeries exactly as in $\mathcal{C}(S).$  Performing the surgeries one at a time, notice that regardless of the type of surgery, the resulting curve will intersect a curve adjacent to it exactly once.  The figures below demonstrates this when the box has three curves involved.  The general case follows exactly the same. See Figure \ref{fig:firstboxsurgeNC} and Figure \ref{fig:secondboxsurgeNC} for the first two steps in the surgery sequence of a box containing the curves 3, 4, and 5.

  \begin{figure}[h!]
        \centering
        \includegraphics[scale=0.8]{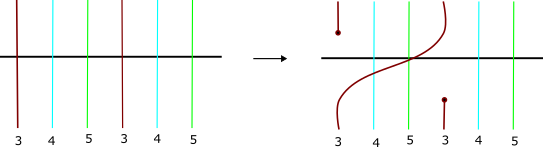}
       \caption{The intersection sequence corresponding to the above box in the dot graph.  Odd surgery on curve 3, intersects curve 4 exactly once, thus it is nonseparating.  Clearly, an even surgery would do the same so there are two surgery options for curve 3 exactly as before.}
      \label{fig:firstboxsurgeNC}
   \end{figure}
 \begin{figure}[h!]
        \centering
      \includegraphics[scale=0.8]{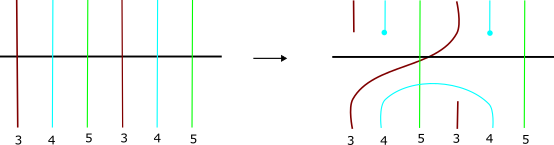}
        \caption{Now surgery on curve 4 is performed. This time an even surgery is demonstrated, as above this curve intersects curve 5 exactly one time so it is nonseparating.  An odd surgery would do the same.}
       \label{fig:secondboxsurgeNC}
    \end{figure}

Continuing in this way, removes the box from the dot graph and replaces all curves with other nonseparating curves.

{\bf Case 2:} Suppose that $\sigma$ is a sequence of natural numbers in sawtooth form and that $G(\sigma)$ has an empty, unpierced hexagon of type 1 or 2.  Then $\sigma$ is reducible. \\

We will treat the case or an empty, unpierced hexagon of type 1.  The other case follows exactly the same procedure.  The surgery instructions for this case are similar to the instructions for $\mathcal{C}(S)$ but one new idea is needed.  We introduce some new terminology to simplify the discussion. Given a empty, unpierced hexagon of type 1 in the dot graph, its vertices have the form:
$$\alpha_k,\dots,\alpha_{k+m},\alpha_k,\dots,\alpha_{k+l},\alpha_{j_1},\dots,\alpha_{j_p},\alpha_{k+l},\dots,\alpha_{k+m}$$
where $1\leq k\leq k+l\leq k+m\leq n-1,p\geq 0,$ and each $j_i<\alpha_{k+l}.$ We will call the integer $l$ the \emph{step length} of the hexagon, the number of vertices in $\alpha_{j_1},\dots,\alpha_{j_p}$ the \emph{tail length} of the hexagon, and the integer $m$ the \emph{total length} of the hexagon.  From the dot graph it is easy to see these values.  For instance the hexagon in \ref{fig:hexpoly} has a  step length 2, and tail length 4.  We will call the curve $\alpha_{k+l}$ \emph{the curve at step length l}.
  \begin{figure}[h!]
        \centering
        \includegraphics[scale=0.8]{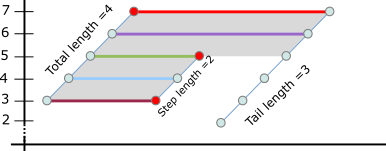}
        \caption{These three numbers along with the starting vertex seen completely determine a type 1 hexagon}
        \label{fig:hexpolylengths}
    \end{figure}

Notice that we may always assume that the tail length of a hexagon is nonzero.  It it were zero, the curve at step length $l$ would occur consecutively in the dot graph, and a trivial surgery on this curve would remove the hexagon from the dot graph. We begin the hexagon surgery the same.  Consider the curve at step length $l,$ this is the only curve that occurs three times in the hexagon.  There exists a surgery on the first two intersection points that removes the third intersection point. This surgery is determined by the orientations of the first two intersection points.  However, whatever surgery is required intersects the curves directly adjacent (above and below) to it exactly one time, thus the result is nonseparating. See figure \ref{fig:firsthexsurgeNC} below: 

 \begin{figure}[h!]
        \centering
        \includegraphics[width=\linewidth,scale=0.8]{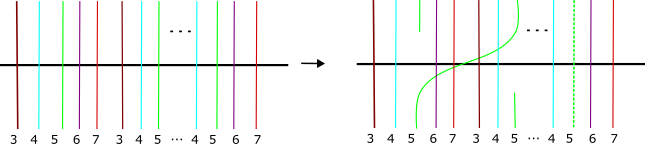}
        \caption{The intersection sequence corresponding to the above hexagon in the dot graph.  Odd surgery on curve 5, intersects curve 6 (and 4) exactly once, thus it is nonseparating.  An even surgery would do the same so any required surgery to delete the last 5 vertex in the dot graph works. The ellipses represent curves in the tail of the hexagon, the curve 4 is seen in the intersection sequence because we assume the tail length is nonzero}
        \label{fig:firsthexsurgeNC}
    \end{figure}

We now attempt to perform surgeries on the curves that occur below the curve at step length $l.$  Just as in the box case, these curves have either surgery available to them since all possible surgeries will intersect the curve directly adjacent to it (above it in the dot graph) exactly once.\\ 

 \begin{figure}[h!]
        \centering
        \includegraphics[width=\linewidth,scale=0.8]{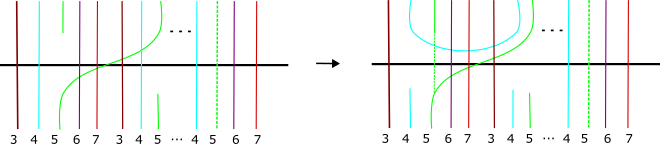}
        \caption{Now surgery on curve 4 is performed. This time an even surgery is demonstrated, as above this curve intersects the old curve 5  exactly one time so it is nonseparating.  An odd surgery would do the same.  The same follows for curve 3.}
        \label{fig:secondhexsurgeNC}
    \end{figure}
Now we are ready perform surgeries on the curves that occur after the curve at step length $l.$  The curve directly adjacent to the curve at step length $l$ (above it in the dot graph) may cause issues.  We break this into two cases:   

{\bf Subcase 2.1:} Let $m$ denote the total length of the hexagon, and let $l$ denote the step length.  If $m > l+1,$ then the surgeries follow exactly as in $\mathcal{C}(S).$  This is clear since all the curves obtained from surgery above the curve at step length $l$ intersect an adjacent curve (above or below in in the dot graph) exactly once. Notice that for the curve directly above the curve at step length $l,$ $\alpha_{k+l+1}$ we use the intersection with the curve above it $\alpha_{k+l+2}$ to show it is nonseparating.  For curves above $\alpha_{k+l+1}$ in the dot graph, say $\alpha_{j}$, we look at the intersection with $\alpha_{j-1}$ to show it is nonseparating.  See figure \ref{fig:thirdhexsurgeNC}below: 
 \begin{figure}[h!]
        \centering
        \includegraphics[width=\linewidth,scale=0.8]{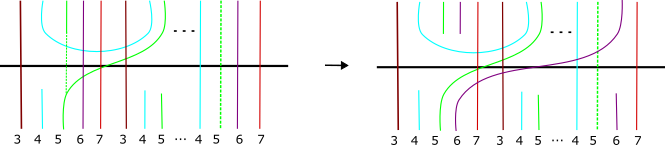}
        \caption{Surgery on curve 6 is performed.  This curve intersects the curve 7  exactly one time so it is nonseparating.  An odd surgery would do the same. Notice the surgery would intersect the old curve 5 twice so that it is not possible to argue that it is nonseparating with curve 5.}
        \label{fig:thirdhexsurgeNC}
    \end{figure}

{\bf Subcase 2.2:} Let $m$ denote the total length of the hexagon, and let $l$ denote the step length.  If $m = l+1,$ then we introduce a new surgery on the curve $\alpha_{k+l+1}.$\\
All the curves below the curve at step length $l$ have the surgeries performed on them as before, each is nonseparating. We now need to perform a surgery on the curve $\alpha_{k+l+1}$ and argue that is nonseparating.  

  \begin{figure}[h!]
        \centering
        \includegraphics[width=\linewidth,scale=0.8]{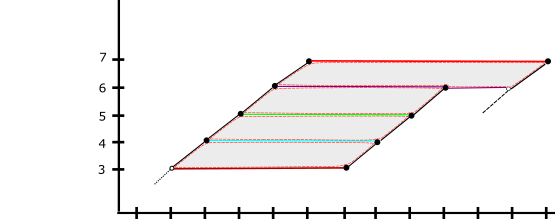}
        \caption{Surgery on the curves 3 through 5 will be just fine, since either type of surgery will intersect the curve directly above it exactly once. But surgery on 7 will intersect 6 twice.}
        \label{fig:specialhexgraphNC}
    \end{figure}

If the curve $\alpha_{k+l+1}$ requires odd surgery, then it is nonseparating because it intersects its old self exactly once as in Proposition \ref{Prop2.6}.  In the case where $\alpha_{k+l+1}$ requires even surgery and it turns out that the even surgery required forces the new curve $\alpha_{k+l+1}'$ to be separating.  We consider the part of $\alpha_{k+l}$ (the curve at step length $l$) ``inside'' of the curve $\alpha_{k+l+1}'$ (We call the ''inside'' of $\alpha_{k+l+1}'$ the part of the surface disjoint from the curve $\alpha_{k+l}'$).   Since we are assuming our hexagon has a tail, we also see the curve $\alpha_{k+l-1}$ one time inside this region.  If we were able to perform even surgery on these two parts of $\alpha_{k+l},$ the resulting curve is nonseparating since it intersects $\alpha_{k+l-1}$ exactly once.  By the assumption that $\alpha_{k+l+1}'$ is separating, the orientation of $\alpha_{k+l}$ inside of $\alpha_{k+l+1}'$ must be consistent with an even surgery, otherwise $\alpha_{k+l}$ would intersect $\alpha_{k+l+1}$.  See the figure \ref{fig:specialhexsurg1NC} below:
  \begin{figure}[h!]
        \centering
        \includegraphics[width=\linewidth,scale=0.8]{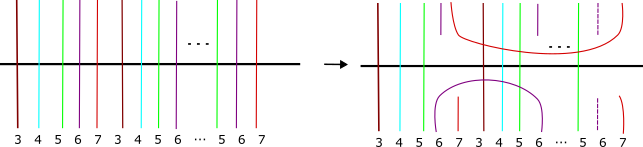}
        \caption{After performing surgery on curve 6, and doing the required even surgery on curve 7 (in figure it is ++), notice the two parts of curve 6 bounded by the new curve 7'.  If we assume that 7' separating then the two parts of 6 inside 7' must be oriented to allow for $++$ surgery. }
        \label{fig:specialhexsurg1NC}
    \end{figure}

 \newpage
Let the new nonseparating curve obtained by joining the ends of curve $\alpha_{k+l}$ be denoted by $\beta.$  We want $\beta$ to be a replacement curve for curve $\alpha_{k+l+1}$ so it must be disjoint from $\alpha_{k+l}'$ and $\alpha_{k+l+2}.$  Clearly, $\beta$ is disjoint from $\alpha_{k+l}'$ since $\alpha_{k+l+1}'$ is disjoint from $\alpha_{k+l}'$, $\beta$ is disjoint from $\alpha_{k+l+1}'$  and $\beta$ is on the other side of $\alpha_{k+l+1}'$ then $\alpha_{k+l}'.$  This is demonstrated in Figure \ref{fig:simplifiedpicreplace}

\begin{figure}[h!]
        \centering
        \includegraphics[]{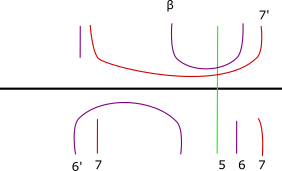}
        \caption{The curve $\beta$ is disjoint from curve 6' since 7' is disjoint from 6' and $\beta$ and  $\beta$ is on the other side of $7'$ than 6'. }
        \label{fig:simplifiedpicreplace}
    \end{figure}

Now we need to argue that the curve $\beta$ is disjoint from the curve $\alpha_{k+l+2}.$  However, because we do not see this curve in our hexagon, this means that $\alpha_{k+l+2}$ is disjoint from $\alpha_{k+l+1}'$ since we knew that it is disjoint from $\alpha_{k+l+1}$ and the curve added to $\alpha_{k+l+1}$ does not intersect $\alpha_{k+l+2}$ otherwise we would have seen the curve in the hexagon.  Thus, if $\alpha_{k+l+2}$ intersects $\beta$ it is contained entirely inside $\alpha_{k+l+1}'$ and therefore disjoint from $\alpha_{k+l}'$ this is a contradiction, since then the curve $\alpha_{k+l+1}$ would not be required in the geodesic.  So, $\beta$ is a nonseparating curve disjoint from $\alpha_{k+l+2}$ and $\alpha_{k+l}'.$  Therefore, $\beta$ is a suitable replacement for $\alpha_{k+l+1}.$ \\

The case for a hexagon of type 2 follows the exact same argument. This covers all the cases, thus proving Lemma \ref{Lemma 2.2}.

\section{Conclusion}

We have demonstrated that efficient geodesics exist in the non-separating curve complex. Moreover, we demonstrated that given any geodesic in the complex of curves
we may obtain an efficient geodesic whose vertices, with the possible exception of the endpoints, are all nonseparating curves.

\subsection{Questions}
Below we have listed a few questions for an interested reader to consider.  
\begin{enumerate}[align=left]
    \item[Question 1.] Can the $n^{6g-6}$ bound be reduced when we restrict to $\mathcal{N}(S)$?\\
     \item[Question 2.] Does replacing a separating curve with a nonseparating curve in a path ever increase the complexity measure?\\
    \item[Question 3.] Find an example of an efficient geodesic that contains a separating curve.  Are there any restrictions on the number of separating curves in a  efficient geodesics? \\
\end{enumerate}

\section{Acknowledgements}
We thank William Menasco for suggesting the problem, as well as many conversations about the efficient geodesic algorithm, flushing out some errors in earlier drafts of this paper, and discussing the new required surgery.

\bibliography{references}  

\vspace{15mm}

\newcommand{\Addresses}{{% additional braces for segregating \footnotesize
  \bigskip
  %\footnotesize

  Seth Hovland, \textsc{Department of Mathematics, University at Buffalo-SUNY,
    Buffalo, NY 14260-2900, USA}\par\nopagebreak
  \textit{E-mail address}: \texttt{sethhovl@buffalo.edu}

  \medskip

 Greg Vinal, \textsc{Department of Mathematics, University at Buffalo-SUNY,
    Buffalo, NY 14260-2900, USA}\par\nopagebreak
  \textit{E-mail address}: \texttt{gregoryv@buffalo.edu}

}}

\Addresses

\end{document}